\documentclass[12pt]{article}
\usepackage{amsfonts}
\usepackage{amsmath}
\usepackage{amssymb}
\usepackage{amsthm}
\usepackage{enumerate}
\makeatletter
\def\imod#1{\allowbreak\mkern10mu({\operator@font mod}\,\,#1)}
\makeatother

\newtheorem{theorem}{Theorem}[section]

\newtheorem{corollary}{Corollary}[section]
\newtheorem{fact}{Fact}[section]

\theoremstyle{definition}
\newtheorem{definition}{Definition}[section]
\newtheorem{remark}{Remark}[section]

\begin{document}
\begin{center}
\vskip 1cm{\LARGE\bf Unitary Multiperfect Numbers in Certain Quadratic Rings 
\vskip 1cm
\large
Colin Defant\footnote{This work was supported by National Science Foundation grant no. 1262930.}\\
Department of Mathematics\\
University of Florida\\
United States\\
cdefant@ufl.edu}
\end{center}
\vskip .2 in

\begin{abstract}
A unitary divisor $c$ of a positive integer $n$ is a positive divisor of $n$ that is relatively prime to $\displaystyle{\frac{n}{c}}$. For any integer $k$, the function $\sigma_k^*$ is a multiplicative arithmetic function defined so that $\sigma_k^*(n)$ is the sum of the $k^{th}$ powers of the unitary divisors of $n$. We provide analogues of the functions $\sigma_k^*$ in imaginary quadratic rings that are unique factorization domains. We then explore properties of what we call $n$-powerfully unitarily $t$-perfect numbers, analogues of the unitary multiperfect numbers that have been defined and studied in the integers. We end with a list of several opportunities for further research. 

\bigskip


\end{abstract}

\section{Introduction} 
We convene to let $\mathbb{N}$ and $\mathbb{P}$ denote the set of positive integers and the set of (integer) prime numbers, respectively. 
\par 
The arithmetic functions $\sigma_k$ are defined, for every integer $k$, by \\ $\displaystyle{\sigma_k(n)=\sum_{\substack{c\vert n\\c>0}}c^k}$. The unitary divisor functions $\sigma_k^*$ are defined by \\ $\displaystyle{\sigma_k^*(n)=\sum_{\substack{0<c\vert n\\ \gcd(c,\frac{n}{c})=1}}c^k}$. In other words, $\sigma_k^*(n)$ is the sum of the $k^{th}$ powers of the unitary divisors of $n$, where a unitary divisor of $n$ is simply a positive divisor $c$ of $n$ such that $c$ and $\displaystyle{\frac{n}{c}}$ are relatively prime. The author has invented and investigated analogues of the divisor functions in imaginary quadratic integer rings that are unique factorization domains \cite{Defant14}. Here, we seek to investigate analogues of the unitary divisor functions in these rings. 
\par 
For any square-free integer $d$, let $\mathcal O_{\mathbb{Q}(\sqrt{d})}$ be the quadratic integer ring given by \[\mathcal O_{\mathbb{Q}(\sqrt{d})}=\begin{cases} \mathbb{Z}[\frac{1+\sqrt{d}}{2}], & \mbox{if } d\equiv 1\imod{4}; \\ \mathbb{Z}[\sqrt{d}], & \mbox{if } d\equiv 2, 3 \imod{4}. \end{cases}\] 
\par 
Throughout the remainder of this paper, we will work in the rings $\mathcal O_{\mathbb{Q}(\sqrt{d})}$ for different specific or arbitrary values of $d$. We will use the symbol ``$\vert$" to mean ``divides" in the ring $\mathcal O_{\mathbb{Q}(\sqrt{d})}$ in which we are working. 
Whenever we are working in a ring other than $\mathbb{Z}$, we will make sure to emphasize when we wish to state that one integer divides another in $\mathbb{Z}$. 
For example, if we are working in $\mathbb{Z}[i]$, the ring of Gaussian integers, we might say that $1+i\vert 1+3i$ and that $2\vert 6$ in $\mathbb{Z}$. We will also refer to primes in $\mathcal O_{\mathbb{Q}(\sqrt{d})}$ as ``primes," whereas we will refer to (positive) primes in $\mathbb{Z}$ as ``integer primes." For an integer prime $p$ and a nonzero integer $n$, we will let $\upsilon_p(n)$ denote the largest integer $k$ such that $p^k\vert n$ in $\mathbb{Z}$. For a prime $\pi$ and a nonzero number $x\!\in\!\mathcal O_{\mathbb{Q}(\sqrt{d})}$, we will let $\rho_\pi(x)$ denote the largest integer $k$ such that $\pi^k\vert x$.  
Furthermore, we will henceforth focus exclusively on values of $d$ for which $\mathcal O_{\mathbb{Q}(\sqrt{d})}$ is a unique factorization domain and $d<0$. In other words, $d\in K$, where we will define $K$ to be the set $\{-163,-67,-43,-19,-11,-7,-3,-2,-1\}$. The set $K$ is known to be the complete set of negative values of $d$ for which $\mathcal O_{\mathbb{Q}(\sqrt{d})}$ is a unique factorization domain \cite{Stark67}.  
\par 
For an element $a+b\sqrt{d}\in\mathcal O_{\mathbb{Q}(\sqrt{d})}$ with $a,b\in \mathbb{Q}$, we define the conjugate by $\overline{a+b\sqrt{d}}=a-b\sqrt{d}$. The norm and absolute value of an element $z$ are defined, respectively, by $N(z)=z\overline{z}$ and $\vert z\vert=\sqrt{N(z)}$. We assume familiarity with the properties of these object, which are treated in Keith Conrad's online notes \cite{Conrad}. For $x,y\in\mathcal O_{\mathbb{Q}(\sqrt{d})}$, we say that $x$ and $y$ are associated, denoted $x\sim y$, if and only if $x=uy$ for some unit $u$ in the ring $\mathcal O_{\mathbb{Q}(\sqrt{d})}$. Furthermore, we will make repeated use of the following well-known facts. 
\begin{fact} \label{Fact1.1}
Let $d\!\in\! K$. If $p$ is an integer prime, then exactly one of the following is true. 
\begin{itemize}
\item $p$ is also a prime in $\mathcal O_{\mathbb{Q}(\sqrt{d})}$. In this case, we say that $p$ is inert in $\mathcal O_{\mathbb{Q}(\sqrt{d})}$. 
\item $p\sim \pi^2$ and $\pi\sim\overline{\pi}$ for some prime $\pi\in \mathcal O_{\mathbb{Q}(\sqrt{d})}$. In this case, we say $p$ ramifies (or $p$ is ramified) in $\mathcal O_{\mathbb{Q}(\sqrt{d})}$. 
\item $p=\pi\overline{\pi}$ and $\pi\not\sim\overline{\pi}$ for some prime $\pi\in\mathcal O_{\mathbb{Q}(\sqrt{d})}$. In this case, we say $p$ splits (or $p$ is split) in $\mathcal O_{\mathbb{Q}(\sqrt{d})}$.
\end{itemize}
\end{fact}
\begin{fact} \label{Fact1.2}
Let $d\!\in\! K$. If $\pi\!\in\!\mathcal O_{\mathbb{Q}(\sqrt{d})}$ is a prime, then exactly one of the following is true. 
\begin{itemize}
\item $\pi\sim q$ and $N(\pi)=q^2$ for some inert integer prime $q$. 
\item $\pi\sim\overline{\pi}$ and $N(\pi)=p$ for some ramified integer prime $p$. 
\item $\pi\not\sim\overline{\pi}$ and $N(\pi)=N(\overline{\pi})=p$ for some split integer prime $p$. 
\end{itemize}
\end{fact}
\begin{fact} \label{Fact1.3}
Let $p$ be an odd integer prime. Then $p$ ramifies in $\mathcal O_{\mathbb{Q}(\sqrt{d})}$ if and only if $p\vert d$ in $\mathbb{Z}$. If $p\nmid d$ in $\mathbb{Z}$, then $p$ splits in $\mathcal O_{\mathbb{Q}(\sqrt{d})}$ if and only if $d$ is a quadratic residue modulo $p$. Note that this implies that $p$ is inert in $\mathcal O_{\mathbb{Q}(\sqrt{d})}$ if and only if $p\nmid d$ in $\mathbb{Z}$ and $d$ is a quadratic nonresidue modulo $p$. 
Also, the integer prime $2$ ramifies in $\mathcal O_{\mathbb{Q}(\sqrt{-1})}$ and $\mathcal O_{\mathbb{Q}(\sqrt{-2})}$, splits in $\mathcal O_{\mathbb{Q}(\sqrt{-7})}$, and is inert in $\mathcal O_{\mathbb{Q}(\sqrt{d})}$ for all $d\in K\backslash\{-1,-2,-7\}$.
\end{fact}
\begin{fact} \label{Fact1.4}
Let $\mathcal O_{\mathbb{Q}(\sqrt{d})}^*$ be the set of units in the ring $\mathcal O_{\mathbb{Q}(\sqrt{d})}$. Then $\mathcal O_{\mathbb{Q}(\sqrt{-1})}^*=\{\pm 1,\pm i\}$, $\displaystyle{\mathcal O_{\mathbb{Q}(\sqrt{-3})}^*=\left\{\pm 1,\pm \frac{1+\sqrt{-3}}{2},\pm \frac{1-\sqrt{-3}}{2}\right\}}$, and $\mathcal O_{\mathbb{Q}(\sqrt{d})}^*=\{\pm 1\}$ \\ whenever $d\in K\backslash\{-1,-3\}$. 
\end{fact}
\par 
For a nonzero complex number $z$, let $\arg (z)$ denote the argument, or angle, of $z$. We convene to write $\arg (z)\in [0,2\pi)$ for all nonzero $z\in\mathbb{C}$. For each $d\in K$, we define the set $A(d)$ by 
\[A(d)=\begin{cases} \{z\in\mathcal O_{\mathbb{Q}(\sqrt{d})} \backslash\{0\}: 0\leq \arg (z)<\frac{\pi}
{2}\}, & \mbox{if } d=-1; \\ \{z\in\mathcal O_{\mathbb{Q}(\sqrt{d})} \backslash\{0\}: 0\leq \arg (z)<\frac{\pi}
{3}\}, & \mbox{if } d=-3; \\ \{z\in\mathcal O_{\mathbb{Q}(\sqrt{d})} \backslash\{0\}: 0\leq \arg (z)<\pi\}, & \mbox{otherwise}. \end{cases}\] 
Thus, every nonzero element of $\mathcal O_{\mathbb{Q}(\sqrt{d})}$ can be written uniquely as a unit times a product of primes in $A(d)$. Also, every $z\in\mathcal O_{\mathbb{Q}(\sqrt{d})}\backslash\{0\}$ is associated to a unique element of $A(d)$. 
For nonzero elements $x,z\in\mathcal O_{\mathbb{Q}(\sqrt{d})}$, we will write $x\Diamond z$ if and only if $x\in A(d)$, $x\vert z$, and $x$ is relatively prime to $\displaystyle{\frac{z}{x}}$ (meaning $x$ and $\displaystyle{\frac{z}{x}}$ have no nonunit common divisors).  
\begin{definition} \label{Def1.1}
Let $d\in K$, and let $n\in\mathbb{Z}$.  Define the function \\ 
$\delta_n^*\colon\mathcal O_{\mathbb{Q}(\sqrt{d})}\backslash\{0\}\rightarrow [1,\infty)$ by 
\[\delta_n^*(z)=\sum_{x\Diamond z}\vert x \vert^n.\]
\end{definition}
\begin{remark} \label{Rem1.1}
We note that, for each $x$ in the summation in the above definition, we may cavalierly replace $x$ with one of its associates. This is because associated numbers have the same absolute value. In other words, the only reason for the criterion $x\!\in\! A (d)$ in the summation that appears in Definition \ref{Def1.1} (which is implied by the relation $x\Diamond z$) is to forbid us from counting associated divisors as distinct terms in the summation, but we may choose to use any of the associated divisors as long as we only choose one. This should not be confused with how we count conjugate divisors (we treat $2+i$ and $2-i$ as distinct divisors of $5$ in $\mathbb{Z}[i]$ because $2+i\not\sim 2-i$). Also, note that the functions $\delta_n^*$ depend on the ring in which we are working (this is also true of the function $I_n^*$, which we will soon define).  
\end{remark}
\par 
We will say that a function $f\colon\mathcal O_{\mathbb{Q}(\sqrt{d})}\backslash\{0\}\!\rightarrow\!\mathbb{R}$ is multiplicative if $f(xy)=f(x)f(y)$ whenever $x$ and $y$ are relatively prime. 
\begin{theorem} \label{Thm1.1}
Let us work in a ring $\mathcal O_{\mathbb{Q}(\sqrt{d})}$ with $d\in K$. For any $n\in\mathbb{Z}$, the function $\delta_n^*$ is multiplicative. 
\end{theorem} 
\begin{proof} 
Let $z_1$ and $z_2$ be relatively prime elements of $\mathcal O_{\mathbb{Q}(\sqrt{d})}\backslash\{0\}$ for some $d\in K$. First note that if $x\Diamond z_1z_2$, then $x\sim x_1x_2$, where $x_1\Diamond z_1$ and $x_2\Diamond z_2$. Furthermore, the numbers $x_1$ and $x_2$ are unique because of the requirement $x_1,x_2\in A(d)$ inherent in the relations $x_1\Diamond z_1$ and $x_2\Diamond z_2$. On the other hand, if $x_1\Diamond z_1$ and $x_2\Diamond z_2$, then $x_1x_2$ is associated to a unique number $x$ such that $x\Diamond z_1z_2$. Therefore, 
\[\delta_n^*(z_1z_2)=\sum_{x\Diamond z_1z_2}\vert x\vert^n=\sum_{\substack{x_1\Diamond z_1 \\ x_2\Diamond z_2}}\vert x_1x_2\vert^n=\sum_{x_1\Diamond z_1}\vert x_1\vert^n\sum_{x_2\Diamond z_2}\vert x_2\vert^n=\delta_n^*(z_1)\delta_n^*(z_2).\]
\end{proof} 
\begin{definition} \label{Def1.2} 
For $d\in K$, define the function $I_n^*\colon\mathcal O_{\mathbb{Q}(\sqrt{d})}\backslash\{0\}\rightarrow[1,\infty)$, for each $n\in\mathbb{Z}$, by $\displaystyle{I_n^*(z)=\frac{\delta_n^*(z)}{\vert z\vert ^n}}$. 
For a positive integer $t\geq 2$, we say that a number $z\in\mathcal O_{\mathbb{Q}(\sqrt{d})}\backslash\{0\}$ is \textit{$n$-powerfully unitarily $t$-perfect in $\mathcal O_{\mathbb{Q}(\sqrt{d})}$} if $I_n^*(z)=t$, and, if $t=2$, we simply say that $z$ is \textit{$n$-powerfully unitarily perfect in $\mathcal O_{\mathbb{Q}(\sqrt{d})}$}. If $n=1$, we will omit the adjective ``$1$-powerfully."   
\end{definition} 
\begin{theorem} \label{Thm1.2}
Let $k,n\in\mathbb{N}$, $d\in K$, and $z\in\mathcal O_{\mathbb{Q}(\sqrt{d})}\backslash\{0\}$. Then, if we are working in the ring $\mathcal O_{\mathbb{Q}(\sqrt{d})}$, the following statements are true. 
\begin{enumerate}[(a)] 
\item The range of $I_n^*$ is a subset of the interval $[1,\infty)$, and $I_n^*(z)=1$ if and only if $z$ is a unit in $\mathcal O_{\mathbb{Q}(\sqrt{d})}$. 
\item $I_n^*$ is multiplicative.  
\item $I_n^*(z)=\delta_{-n}^*(z)$. 
\end{enumerate}  	
\end{theorem} 
\begin{proof} 
Part $(a)$ is fairly trivial. To prove part $(b)$, let $z_1$ and $z_2$ be relatively prime elements of $\mathcal O_{\mathbb{Q}(\sqrt{d})}$. We may use Theorem \ref{Thm1.1} to write 
\[I_n^*(z_1z_2)=\frac{\delta_n^*(z_1z_2)}{\vert z_1z_2\vert^n}=\frac{\delta_n^*(z_1)\delta_n^*(z_2)}{\vert z_1\vert^n\vert z_2\vert^n}=I_n^*(z_1)I_n^*(z_2).\]
To prove part $(c)$, it suffices, due to the truth of part $(b)$, to show that $I_n^*(\pi^{\alpha})=\delta_{-n}^*(\pi^{\alpha})$ for an arbitrary prime $\pi$ and positive integer $\alpha$. We have 
\[I_n^*(\pi^{\alpha})=\frac{\delta_n^*(\pi^{\alpha})}{\vert \pi^{\alpha}\vert^n}=\vert \pi\vert^{-\alpha n}\sum_{x\Diamond\pi^{\alpha}}\vert x\vert^n=\vert \pi\vert^{-\alpha n}(1+\vert\pi^{\alpha}\vert^n)\]
\[=1+\vert\pi^{\alpha}\vert^{-n}=\sum_{x\Diamond\pi^{\alpha}}\vert x\vert^{-n}=\delta_{-n}^*(\pi^{\alpha}).\]
\end{proof} 
\begin{remark} \label{Rem1.2} 
Let $d\in K$, and let $z\in\mathcal O_{\mathbb{Q}(\sqrt{d})}\backslash\{0\}$ satisfy $\displaystyle{z\sim\prod_{j=1}^r\pi_j^{\alpha_j}}$, where, for all distinct $j,\ell\in\{1,2,\ldots,r\}$, $\pi_j$ is a prime, $\alpha_j$ is a positive integer, and $\pi_j\not\sim\pi_{\ell}$. Combining parts $(b)$ and $(c)$ of Theorem \ref{Thm1.2}, we see that, for any positive integer $n$, we may calculate $I_n^*(z)$ as $\displaystyle{I_n^*(z)=\prod_{j=1}^r(1+\vert\pi_j\vert^{-\alpha_j n})}$. 
\end{remark}
As an example, let us calculate $I_2^*(30)$ in $\mathcal O_{\mathbb{Q}(\sqrt{-1})}$. We have \[30\sim(1+i)^2\cdot 3(2+i)(2-i),\] so \[I_2^*(30)=I_2^*\left((1+i)^2\right)I_2^*(3)
I_2^*(2+i)I_2^*(2-i)\] 
\[=\left(1+\frac{1}{N(1+i)^2}\right)\left(1+\frac{1}{N(3)}\right)\left(1+\frac{1}{N(2+i)}\right)\left(1+\frac{1}{N(2-i)}\right)\] 
\[=\frac{5}{4}\cdot\frac{10}{9}\cdot\frac{6}{5}\cdot\frac{6}{5}=2.\] Thus, $30$ is $2$-powerfully unitarily perfect in $\mathcal O_{\mathbb{Q}(\sqrt{-1})}$.  
\par 
Now that we have established the foundations that we will need, we may study the properties of some $n$-powerfully unitarily $t$-perfect numbers. 
\section{Investigating $n$-powerfully Unitarily \\ $t$-perfect Numbers} 
\begin{theorem} \label{Thm2.1} 
Let $d\in K$, and let $z\in\mathcal O_{\mathbb{Q}(\sqrt{d})}\backslash\{0\}$. For any integer $n\geq 4$, $I_n^*(z)<2$. Furthermore, if $I_3^*(z)$ is rational, then $I_3^*(z)<2$. 
\end{theorem}
\begin{proof} 
Let $\Psi(z)$ be the set of all primes in $A(d)$ that divide $z$, and let $\Phi$ be the set of all primes in $A(d)$. Then, for any integer $n\geq 3$, 
\[I_n^*(z)=\prod_{\pi\in\Psi(z)}(1+\vert\pi\vert^{-\rho_{\pi}(z)n})<\prod_{\pi\in\Psi(z)}(1+\vert\pi\vert^{-n})<\prod_{\pi\in\Phi}(1+\vert\pi\vert^{-n})\]
\[=\prod_{\substack{\pi\in\Phi\\ \vert\pi\vert\in\mathbb{N}}}(1+\vert\pi\vert^{-n}) \prod_{\substack{\pi\in\Phi\\ \vert\pi\vert\not\in\mathbb{N} \\ \pi\sim\overline{\pi}}}(1+\vert\pi\vert^{-n}) \prod_{\substack{\pi\in\Phi\\ \vert\pi\vert\not\in\mathbb{N} \\ \pi\not\sim\overline{\pi}}}(1+\vert\pi\vert^{-n})\]
\[=\prod_{\substack{q\in\mathbb{P}\\ q\hspace{0.75 mm} is\hspace{0.75 mm} inert}}(1+q^{-n})
\prod_{\substack{p\in\mathbb{P}\\ p\hspace{0.75 mm} ramifies}}(1+\sqrt{p}^{-n})
\prod_{\substack{p\in\mathbb{P}\\ p\hspace{0.75 mm} splits}}(1+\sqrt{p}^{-n})^2.\]
If $n\geq 5$, then we have 
\[I_n^*(z)<\prod_{\substack{q\in\mathbb{P}\\ q\hspace{0.75 mm} is\hspace{0.75 mm} inert}}(1+q^{-n})
\prod_{\substack{p\in\mathbb{P}\\ p\hspace{0.75 mm} ramifies}}(1+\sqrt{p}^{-n})
\prod_{\substack{p\in\mathbb{P}\\ p\hspace{0.75 mm} splits}}(1+\sqrt{p}^{-n})^2\]
\[<\prod_{\substack{q\in\mathbb{P}\\ q\hspace{0.75 mm} is\hspace{0.75 mm} inert}}(1+\sqrt{q}^{-n})^2
\prod_{\substack{p\in\mathbb{P}\\ p\hspace{0.75 mm} ramifies}}(1+\sqrt{p}^{-n})^2
\prod_{\substack{p\in\mathbb{P}\\ p\hspace{0.75 mm} splits}}(1+\sqrt{p}^{-n})^2\]
\[=\prod_{p\in\mathbb{P}}(1+\sqrt{p}^{-n})^2\leq\prod_{p\in\mathbb{P}}\left(1+\sqrt{p}^{-5}\right)^2=\prod_{p\in\mathbb{P}}\left(\frac{1-p^{-5}}{1-\sqrt{p}^{-5}}\right)^2\] 
\[=\left(\frac{\zeta(5/2)}{\zeta(5)}\right)^2<2,\]
where $\zeta$ denotes the Riemann zeta function. 
\par 
Next, suppose $n=4$. Let us assume that $d\neq -7$ so that $2$ does not split in $\mathcal O_{\mathbb{Q}(\sqrt{d})}$. Then 
\[I_4^*(z)<\prod_{\substack{q\in\mathbb{P}\\ q\hspace{0.75 mm} is\hspace{0.75 mm} inert}}\left(1+q^{-4}\right)
\prod_{\substack{p\in\mathbb{P}\\ p\hspace{0.75 mm} ramifies}}\left(1+\sqrt{p}^{-4}\right)
\prod_{\substack{p\in\mathbb{P}\\ p\hspace{0.75 mm} splits}}\left(1+\sqrt{p}^{-4}\right)^2\]
\[<\prod_{\substack{p\in\mathbb{P}\\ p\hspace{0.75 mm} does \\ not\hspace{0.75 mm} split}}\left(1+\sqrt{p}^{-4}\right)
\prod_{\substack{p\in\mathbb{P}\\ p\hspace{0.75 mm} splits}}\left(1+\sqrt{p}^{-4}\right)^2\] 
\[=\left(1+\sqrt{2}^{-4}\right)\prod_{\substack{p\in\mathbb{P}\backslash\{2\}\\ p\hspace{0.75 mm} does \\ not\hspace{0.75 mm} split}}\left(1+\sqrt{p}^{-4}\right)
\prod_{\substack{p\in\mathbb{P}\\ p\hspace{0.75 mm} splits}}\left(1+\sqrt{p}^{-4}\right)^2\]
\[<\left(1+\sqrt{2}^{-4}\right)\prod_{p\in\mathbb{P}\backslash\{2\}}\left(1+\sqrt{p}^{-4}\right)^2=\left(1+\sqrt{2}^{-4}\right)^{-1}\prod_{p\in\mathbb{P}}\left(1+\sqrt{p}^{-4}\right)^2\]
\[=\frac{4}{5}\prod_{p\in\mathbb{P}}\left(1+p^{-2}\right)^2=\frac{4}{5}\prod_{p\in\mathbb{P}}\left(\frac{1-p^{-4}}{1-p^{-2}}\right)^2=\frac{4}{5}\left(\frac{\zeta(2)}{\zeta(4)}\right)^2<2.\]
Now, assume that $d=-7$ so that $3$ is inert. We then have  
\[I_4^*(z)<\prod_{\substack{q\in\mathbb{P}\\ q\hspace{0.75 mm} is\hspace{0.75 mm} inert}}\left(1+q^{-4}\right)
\prod_{\substack{p\in\mathbb{P}\\ p\hspace{0.75 mm} ramifies}}\left(1+\sqrt{p}^{-4}\right)
\prod_{\substack{p\in\mathbb{P}\\ p\hspace{0.75 mm} splits}}\left(1+\sqrt{p}^{-4}\right)^2\] 
\[<\left(1+3^{-4}\right)\prod_{p\in\mathbb{P}\backslash\{3\}}\left(1+\sqrt{p}^{-4}\right)^2=\frac{1+3^{-4}}{\left(1+\sqrt{3}^{-4}\right)^2}\prod_{p\in\mathbb{P}}\left(1+\sqrt{p}^{-4}\right)^2\]
\[=\frac{41}{50}\prod_{p\in\mathbb{P}}\left(\frac{1-p^{-4}}{1-p^{-2}}\right)^2=\frac{41}{50}\left(\frac{\zeta(2)}{\zeta(4)}\right)^2<2.\]
\par  
Finally, suppose $n=3$ and $I_3^*(z)$ is rational. If $\pi$ is a prime and $\vert\pi\vert=\sqrt{p}$ for some integer prime $p$, then it is easy to see that $\rho_{\pi}(z)$ must be even in order for $I_3^*(z)$ to be rational. Therefore, 
\[I_3^*(z)=\prod_{\pi\in\Psi(z)}\left(1+\vert\pi\vert^{-3\rho_{\pi}(z)}\right)\] 
\[<\prod_{\substack{q\in\mathbb{P}\\ q\hspace{0.75 mm} is\hspace{0.75 mm} inert}}\left(1+q^{-3}\right)
\prod_{\substack{p\in\mathbb{P}\\ p\hspace{0.75 mm} ramifies}}\left(1+\sqrt{p}^{-6}\right)
\prod_{\substack{p\in\mathbb{P}\\ p\hspace{0.75 mm} splits}}\left(1+\sqrt{p}^{-6}\right)^2\] 
\[<\prod_{p\in\mathbb{P}}\left(1+p^{-3}\right)^2=\prod_{p\in\mathbb{P}}\left(\frac{1-p^{-6}}{1-p^{-3}}\right)^2=\left(\frac{\zeta(3)}{\zeta(6)}\right)^2<2.\] 
\end{proof}  
\begin{corollary} \label{Cor2.1} 
If $n\geq 3$ and $t\geq 2$ are integers, then there are no $n$-powerfully unitarily $t$-perfect numbers in any ring $\mathcal O_{\mathbb{Q}(\sqrt{d})}$ with $d\in K$. 
\end{corollary}
\begin{theorem} \label{Thm2.2}
Let us work in a ring $\mathcal O_{\mathbb{Q}(\sqrt{d})}$ with $d\in K$. Suppose $z\in\mathcal O_{\mathbb{Q}(\sqrt{d})}$ satisfies $I_n^*(z)=t$ for some $n\in\{1,2\}$ and $t\in\mathbb{N}\backslash\{1\}$. Then $N(z)$ is even.   
\end{theorem} 
\begin{proof} 
Assume, for the sake of finding a contradiction, that $N(z)$ is odd. Write $\displaystyle{z\sim\prod_{i=1}^r\pi_j^{\alpha_j}}$, where, for all distinct $j,\ell\in\{1,2,\ldots,r\}$, $\pi_j$ is a prime, $\alpha_j$ is a positive integer, and $\pi_j\not\sim\pi_{\ell}$. Suppose that $n=1$. If $N(\pi_j)$ is an integer prime for some $j\in\{1,2,\ldots,r\}$, then it is easy to see that $\alpha_j$ must be even in order for $\displaystyle{I_1^*(z)=\prod_{i=1}^r(1+\vert\pi_j\vert^{-\alpha_j})}$ to be an integer (or even a rational number). This means that $\vert\pi_j\vert^{\alpha_j}$ is an integer for each $j\in\{1,2,\ldots,r\}$, so $\displaystyle{\delta_1^*(z)=\prod_{i=1}^r(1+\vert\pi_j\vert^{\alpha_j})}$ and $\displaystyle{\vert z\vert=\prod_{i=1}^r\vert\pi_j\vert^{\alpha_j}}$ are positive integers. Furthermore, $\vert\pi_j\vert^{\alpha_j}$ must be odd for each $j\in\{1,2,\ldots,r\}$, so $2^r\vert\delta_1^*(z)$ in $\mathbb{Z}$. As $\delta_1^*(z)=t\vert z\vert$ and $\vert z\vert$ is odd, we see that $2^r\vert t$ in $\mathbb{Z}$. However, $\displaystyle{t=I_1^*(z)=\prod_{i=1}^r(1+\vert\pi_j\vert^{-\alpha_j})\leq\prod_{i=1}^r\left(1+\frac{1}{3}\right)}=\left(\frac{4}{3}\right)^r$, which a contradiction.  
\par 
Now, suppose $n=2$. Then $\delta_2^*(z)$ and $N(z)$ are positive integers. Because $\displaystyle{\delta_2^*(z)=\prod_{i=1}^r(1+N(\pi_j)^{\alpha_j})}$ and $N(\pi_j)^{\alpha_j}$ is odd for each $j\in\{1,2,\ldots,r\}$, we see that $2^r\vert\delta_2^*(z)$ in $\mathbb{Z}$. Again, $2^r\vert t$ in $\mathbb{Z}$, which is a contradiction because $\displaystyle{t=I_2^*(z)=\prod_{i=1}^r(1+N(\pi_j)^{-\alpha_j})\leq\prod_{i=1}^r\left(1+\frac{1}{3}\right)}=\left(\frac{4}{3}\right)^r$. 
\end{proof} 
The rings $\mathcal O_{\mathbb{Q}(\sqrt{-1})}$ and $\mathcal O_{\mathbb{Q}(\sqrt{-3})}$ are two of the most heavily-studied \\ quadratic rings, so it is not surprising that they prove to be particularly interesting for our purposes. We proceed to prove a theorem about $2$-powerfully $t$-perfect numbers in each of these rings. 
\begin{theorem} \label{Thm2.3} 
Suppose $z$ is $2$-powerfully unitarily $t$-perfect in $\mathcal O_{\mathbb{Q}(\sqrt{-1})}$ for some integer $t\geq 2$. Then we may write $z=(1+i)^{\gamma}x$, where $x\in\mathcal O_{\mathbb{Q}(\sqrt{-1})}$ and $N(x)$ is odd. Also, $x$ has $\gamma+\upsilon_2(t)$ nonassociated prime divisors.  
\end{theorem} 
\begin{proof} 
Let us write $\displaystyle{x\sim\prod_{j=1}^r\pi_j^{\alpha_j}}$, where, for all distinct $j,\ell\in\{1,2,\ldots,r\}$, $\pi_j$ is a prime, $\alpha_j$ is a positive integer, and $\pi_j\not\sim\pi_{\ell}$. From Fact \ref{Fact1.3}, we know that an integer prime is inert in $\mathcal O_{\mathbb{Q}(\sqrt{-1})}$ if and only if it is congruent to $3$ modulo $4$. Therefore, if we choose any $j\in\{1,2,\ldots,r\}$, then either $N(\pi_j)=q^2$ for some integer prime $q$ that is congruent to $3$ modulo $4$ or $N(\pi_j)=p$ for some integer prime $p$ that is congruent to $1$ modulo $4$. Either way, $N(\pi_j)\equiv 1\imod{4}$, so $\displaystyle{\upsilon_2(\delta_2^*(x))=\upsilon_2\left(\prod_{i=1}^r(1+N(\pi_j)^{\alpha_j})\right)=r}$. Then the desired result follows from the equation $(2^{\gamma}+1)\delta_2^*(x)=2^{\gamma}tN(x)$ and the fact that $\upsilon_2(N(x))=\upsilon_2(2^{\gamma}+1)=0$.  
\end{proof}
\begin{theorem} \label{Thm2.4} 
Let us work in the ring $\mathcal O_{\mathbb{Q}(\sqrt{-3})}$. If $a$ and $b$ are relatively prime positive integers and $3\vert a$ in $\mathbb{Z}$, then $\displaystyle{\frac{a}{b}}$ is not in the range of the function $I_2^*$. 
\end{theorem} 
\begin{proof} 
For the sake of finding a contradiction, suppose $\displaystyle{I_2^*(z)=\frac{a}{b}}$
for some $z\in\mathcal O_{\mathbb{Q}(\sqrt{-3})}$. Then $b\delta_2^*(z)=aN(z)$, which implies that $3\vert \delta_2^*(z)$ in $\mathbb{Z}$. This means that there must be some prime $\pi_0$ such that $N(\pi_0)^{\rho_{\pi_0}(z)}\equiv 2\imod{3}$. Fact \ref{Fact1.3} tells us that an integer prime is inert in $\mathcal O_{\mathbb{Q}(\sqrt{-3})}$ if and only if it is congruent to $2$ modulo $3$. If $N(\pi_0)=q^2$ for some inert integer prime $q$, then $N(\pi_0)^{\rho_{\pi_0}(z)}=q^{2\rho_{\pi_0}(z)}\equiv 1\imod{3}$, which is a contradiction. Clearly $\pi_0\not\sim 3$, so $N(\pi_0)$ must be a split integer prime. However, this means that $N(\pi_0)\equiv 1\imod{3}$, so $N(\pi_0)^{\rho_{\pi_0}(z)}\equiv 1\imod{3}$, which is a contradiction.  
\end{proof} 
\begin{corollary} \label{Cor2.2} 
If $t$ is a positive integer multiple of $3$, then there are no $2$-powerfully unitarily $t$-perfect numbers in $\mathcal O_{\mathbb{Q}(\sqrt{-3})}$. 
\end{corollary} 
Now, let us work in rings $\mathcal O_{\mathbb{Q}(\sqrt{d})}$ with $d\in K\backslash\{-7\}$ so that $2$ does not split. Then there is a unique prime $\xi(d)\in\mathcal O_{\mathbb{Q}(\sqrt{d})}\cap A(d)$ of minimal even norm. Namely, 
\[\xi(d)=\begin{cases} 1+i, & \mbox{if } d=-1; \\ \sqrt{-2}, & \mbox{if } d=-2; \\ 2, & \mbox{if } d\in K\backslash\{-1,-2,-7\}. \end{cases}\]
Suppose $z$ is a $2$-powerfully unitarily $t$-perfect number in $\mathcal O_{\mathbb{Q}(\sqrt{d})}$ for $d\in K\backslash\{-7\}$ and $t\in\mathbb{N}\backslash\{1\}$. By Theorem \ref{Thm2.1}, we see that we may write $z\sim(\xi(d))^{\mu}x_0$, where $\mu\in\mathbb{N}$, $x_0\in\mathcal O_{\mathbb{Q}(\sqrt{d})}$, and $2\nmid N(x_0)$ in $\mathbb{Z}$. Furthermore, if $d\in\{-1,-2\}$, then we have $2^{\mu}+1=\delta_2^*((\xi(d))^{\mu})\vert\delta_2^*(z)=tN(z)$. Hence, if we assume that $3\nmid N(z)$ in $\mathbb{Z}$, then $\mu$ must be even. Therefore, under the assumption that $3\nmid N(z)$ in $\mathbb{Z}$, we may write 
\[\gamma=\begin{cases} \frac{1}{2}\mu, & \mbox{if } d\in\{-1,-2\}; \\ \mu, & \mbox{if } d\in K\backslash\{-1,-2,-7\} \end{cases}\]
so that $z\sim 2^{\gamma}x_0$. Then $z=2^{\gamma}x$, where $x$ is an associate of $x_0$.  
\par 
When M. V. Subbarao and L. J. Warren studied unitary perfect numbers, which are positive integers $n$ that satisfy $\sigma^*(n)=2n$, they noticed that all known unitary perfect numbers are multiples of $3$. They then gave four conditions that any unitary perfect numbers not divisible by $3$ would need to satisfy \cite{Subbarao}. Using the information discussed in the preceding paragraph, we will find analogues of the conditions that Subbarao and Warren established. 
\begin{theorem} \label{Thm2.5} 
Let $d\in K\backslash\{-7\}$. Suppose $z$ is $2$-powerfully perfect in $\mathcal O_{\mathbb{Q}(\sqrt{d})}$ and $3\nmid N(z)$ in $\mathbb{Z}$. Then we may write $z=2^{\gamma}x$, where $\gamma\in\mathbb{N}$, $x\in\mathcal O_{\mathbb{Q}(\sqrt{d})}$, and $N(x)$ is odd. For any prime $\pi$, we have $N(\pi)^{\rho_{\pi}(x)}\equiv 1\imod{6}$. Furthermore, there exists a prime divisor $\pi_0$ of $x$ such that $N(\pi_0)\equiv 5\imod{6}$, and $x$ has an even number of nonassociated prime factors. 
\end{theorem}
\begin{proof} 
We already established that we may write $z=2^{\gamma}x$ for $\gamma\in\mathbb{N}$. As $\delta_2^*(2^{\gamma})=2^{2\gamma}+1$ and $N(2^\gamma)=2^{2\gamma}$, we see that $(2^{2\gamma}+1)\delta_2^*(x)=2^{2\gamma+1}N(x)$. Now, let $\pi$ be a prime. We wish to show that $N(\pi)^{\rho_{\pi}(x)}\equiv 1\imod{6}$. The result is clear if $\rho_{\pi}(x)=0$, and if $\rho_{\pi}(x)>0$, the result is still quite trivial when we consider that $1+N(\pi)^{\rho_{\pi}(x)}\vert\delta_2^*(x)$ in $\mathbb{Z}$. The fact that there exists some prime divisor $\pi_0$ of $x$ such that $N(\pi_0)\equiv 5\imod{6}$ follows from the fact that $2^{2\gamma}+1\equiv 5\imod{6}$. Finally, to show that $x$ has an even number of nonassociated prime divisors, we use the fact that $N(\pi)^{\rho_{\pi}(x)}\equiv 1\imod{6}$ for all primes $\pi$. This implies that $N(x)\equiv 1\imod{3}$. As $2^{2\gamma+1}\equiv 2^{2\gamma}+1\equiv 2\imod{3}$, we see that $\delta_2^*(x)\equiv 1\imod{3}$. Let us write $\displaystyle{x\sim\prod_{j=1}^r\pi_j^{\alpha_j}}$, where, for all distinct $j,\ell\in\{1,2,\ldots,r\}$, $\pi_j$ is a prime, $\alpha_j$ is a positive integer, and $\pi_j\not\sim\pi_{\ell}$. Then $\displaystyle{\delta_2^*(x)=\prod_{j=1}^r(1+N(\pi_j)^{\alpha_j}})\equiv\prod_{j=1}^r(2)\imod{3}$, so $r$ must be even. 
\end{proof}
We pause to mention that we may easily establish results analogous to those given in Theorem \ref{Thm2.5} in the ring $\mathcal O_{\mathbb{Q}(\sqrt{-7})}$. In this ring, $2$ splits as $2=\varepsilon\overline{\varepsilon}$, where $\displaystyle{\varepsilon=\frac{1+\sqrt{-7}}{2}}$. Suppose that $z$ is $2$-powerfully unitarily perfect in $\mathcal O_{\mathbb{Q}(\sqrt{-7})}$ and that $3\nmid N(z)$ in $\mathbb{Z}$. Then we may write $z=\varepsilon^{\gamma_1}\overline{\varepsilon}^{\gamma_2}x$, where $x\in\mathcal O_{\mathbb{Q}(\sqrt{-7})}$ and $N(x)$ is odd. If $\gamma_1\neq 0$ and $\gamma_2\neq 0$, then $(2^{\gamma_1}+1)(2^{\gamma_2}+1)\delta_2^*(x)=2^{\gamma_1+\gamma_2+1}N(x)$. On the other hand, if $\gamma_1=0$ or $\gamma_2=0$ ($\gamma_1$ and $\gamma_2$ cannot both be $0$ by Theorem \ref{Thm2.1}), then we may write $\gamma=\gamma_1+\gamma_2$ to get $(2^{\gamma}+1)\delta_2^*(x)=2^{\gamma+1}N(x)$. Because $3\nmid N(x)$ in $\mathbb{Z}$, we know that $\gamma_1$ and $\gamma_2$ must be even and that $N(\pi)^{\rho_{\pi}(x)}\equiv 1\imod{6}$ for all primes $\pi$. Furthermore, because $2^{\gamma_1}+1\equiv 2^{\gamma_2}+1\equiv 5\imod{6}$, we see that $x$ must have some prime divisor whose norm is congruent to $5$ modulo $6$. Finally, if $\gamma_1\neq 0$ and $\gamma_2\neq 0$, then $(2^{\gamma_1}+1)(2^{\gamma_2}+1)\equiv 1\imod{3}$ and $2^{\gamma_1+\gamma_2+1}N(x)\equiv 2\imod{3}$, so $x$ must have an odd number of nonassociated prime divisors. If $\gamma_1=0$ or $\gamma_2=0$, then $x$ must have an even number of nonassociated prime divisors because $2^{\gamma}+1\equiv 2^{\gamma+1}N(x)\equiv 2\imod{3}$. 
\par 
We end with a note about unitarily $t$-perfect numbers. If $d\in K$ and $t\geq 2$ is an integer, then we can find a unitarily $t$-perfect number in $\mathcal O_{\mathbb{Q}(\sqrt{d})}$ for every unitary $t$-perfect number in $\mathbb{Z}$. We formalize and generalize this notion in the following theorem.
\begin{theorem} \label{Thm2.6} 
Let $b>1$ be a rational number, and let $d\in K$. Let $U(b)=\{n\in\mathbb{N}\colon\sigma^*(n)=bn\}$, and let $V_d(b)=\{z\in A(d)\colon I_1^*(z)=b\}$. Then there exists an injective function $g\colon U(b)\rightarrow V_d(b)$. 
\end{theorem} 
\begin{proof} 
If $p$ is an integer prime that does not split in $\mathcal O_{\mathbb{Q}(\sqrt{d})}$, let $g(p)=p$. If $p$ is an integer prime that splits in $\mathcal O_{\mathbb{Q}(\sqrt{d})}$ as $p=\pi\overline{\pi}$, where $\pi\in A(d)$, let $g(p)$ be the associate of $\pi^2$ in $A(d)$. Now, for any positive integer $n\in U(b)$ with canonical prime factorization $\displaystyle{n=\prod_{j=1}^r p_j^{\alpha_j}}$, let $g(n)$ be the associate of $\displaystyle{\prod_{j=1}^r g(p_j)^{\alpha_j}}$ that lies in $A(d)$. It is easy to see, using the fact that $\mathcal O_{\mathbb{Q}(\sqrt{d})}$ is a unique factorization domain, that $g$ is an injection. To show that the range of $g$ is a subset of $V_d(b)$, note that $\vert g(p)\vert=p$ for all primes $p$. Therefore, with $n$ as before, we have
\[I_1^*(g(n))=I_1^*\left(\prod_{j=1}^r g(p_j)^{\alpha_j}\right)=\prod_{j=1}^r\left(1+\vert g(p_j)\vert^{-\alpha_j}\right)\] 
\[=\prod_{j=1}^r\left(1+p_j^{-\alpha_j}\right)=\frac{\sigma^*(n)}{n}=b.\]  
\end{proof} 
\section{Ideas for Further Research} 
With Theorem \ref{Thm1.2} as evidence, we see that the functions $\delta_n^*$ and $I_n^*$ have some fairly nice properties that we may exploit for further research. We pose some ideas here. 
\par 
First, we note that we could generalize the ideas presented in this paper to other quadratic rings. However, if we choose to continue working with imaginary quadratic rings that are unique factorization domains, we could still look at analogues of many other objects defined in the integers. For example, one might wish to investigate analogues of superperfect numbers and unitary superperfect numbers. One could also look at analogues of biunitary or even infinitary divisor functions in quadratic rings. 
\par 
There are also plenty of questions left open related to the ideas discussed in this paper. For example, the author has made no attempt to actually find $n$-powerfully unitarily $t$-perfect numbers, so it is likely that many could be quite easy to discover. One question of particular interest is the following. For a given $d\in K$, what are the rational numbers $b>1$ for which the function $g\colon U(b)\rightarrow V_d(b)$ defined in the proof of Theorem \ref{Thm2.6} is bijective?

\end{document}